%% file: preprint_v3.tex
\documentclass[11pt]{article}

\usepackage[margin=20truemm]{geometry}

\usepackage{natbib}

\usepackage{amsmath, amssymb, bm}
\usepackage{mathrsfs}

\usepackage{amsfonts,amsthm}
\usepackage{graphicx}
\usepackage{color}
\usepackage{braket}
\usepackage{txfonts}
\usepackage{appendix, etoolbox}

\input{math}

\theoremstyle{plain}
\newtheorem{theorem}{Theorem}
\newtheorem{lemma}[theorem]{Lemma}

\newtheorem{proposition}[theorem]{Proposition}
\theoremstyle{definition}

\theoremstyle{remark}

\begin{document}

\title{Existence of Firth's modified estimates in binomial regression models}
\author{
Mitsunori Ogawa\thanks{Interfaculty Initiative in Information Studies, The University of Tokyo} \
and Yui Tomo \thanks{Department of Biostatistics, The University of Tokyo} \thanks{Department of Clinical Data Science, National Center of Neurology and Psychiatry} \thanks{Department of Health Policy and Management, Keio University}\
}
\date{}
\maketitle

\begin{abstract}
In logistic regression modeling, Firth's modified estimator is widely used to address the issue of data separation, which results in the nonexistence of the maximum likelihood estimate.
Firth's modified estimator can be formulated as a penalized maximum likelihood estimator in which Jeffreys' prior is adopted as the penalty term.
Despite its widespread use in practice, the formal verification of the corresponding estimate's existence has not been established.
In this study, we establish the existence theorem of Firth's modified estimate in binomial logistic regression models, assuming only the full column rankness of the design matrix.
We also discuss other binomial regression models obtained through alternating link functions and prove the existence of similar penalized maximum likelihood estimates for such models.
\end{abstract}

\section{Introduction}
\label{sec:intro}

The logistic regression model is one of the most fundamental models in generalized linear models, and maximum likelihood estimation is a standard method for parameter estimation.
However, it is widely known that the maximum likelihood estimate in a logistic regression model does not exist when data separation occurs (\cite{Albert1984-mk}).
Roughly speaking, this happens when a separation hyperplane exists that separates the data points according to their categories.
Firth's method is often used to address this issue (\cite{Firth1993-li, Heinze2002-je}).
In practice, it has been applied in various fields, including social sciences (\cite{Bandyopadhyay2013-pf, Bhavnani2009-et}) and medical sciences (\cite{Chatsirisupachai2021-td, Bambauer2006-da, Teoh2020-yc}).

Firth's method was originally proposed as a bias correction method for maximum likelihood estimators in \cite{Firth1993-li}.
In this method, the asymptotically bias-corrected estimator is defined as the solution to the modified score equation.
It can also be formulated as a penalized maximum likelihood estimator, in which Jeffreys' prior is adopted as the penalty term.
In both formulations, Firth's method outputs a bias-corrected estimate directly, without computing the original maximum likelihood estimate.
This is a distinctive feature not found in typical bias correction methods, where the maximum likelihood estimate is corrected using a bias correction term.
This feature allows us to expect Firth's modified estimate to exist even when the maximum likelihood estimate does not exist.
Some theoretical properties of the Jeffreys-prior penalty term have been investigated; for example, see \cite{Chen2008-cx, Kosmidis2009-mb, Kosmidis2021-bk}.

Firth's method has been applied to various regression models to obtain an estimator for which the corresponding estimate is ensured to exist even when the original one does not.
\cite{Heinze2002-je} and \cite{Heinze2001-uq} applied Firth's method to the binomial logistic regression model and the Cox proportional hazards model, respectively.
\cite{Bull2002-ed} considered the multinomial logistic regression model and proposed a similar solution to the separation problem.
\cite{Joshi2022-dc} discussed the Poisson regression case and explored some modified estimation methods.
Additionally, \cite{Alam2022-ny} investigated the application of Firth's method to accelerated failure time models.

Despite the widespread use of Firth's method in methodological studies of regression models and their applications in practice, the validity of the existence of Firth's modified estimate has not been fully verified.
\cite{Firth1993-li} yielded an intuitive argument on the existence of Firth's estimate in the binomial logistic regression model, which claimed that the Jeffreys-prior penalty term is unbounded below as the parameter diverges without any formal proof.
\cite{Heinze2002-je} demonstrated the validity of existence through extensive numerical experiments.
Recently, \cite{Kosmidis2021-bk} discussed the theoretical properties of Firth's estimate in binomial regression models.
They examined the finiteness through the examination of the divergence of the Jeffreys-prior penalty term.
However, their discussion implicitly assumed the existence in a vague sense and was insufficient as formal proof.
Overall, although the existence of Firth's modified estimate is empirically evident, more formal theoretical justification is still required.

To fill this gap, we establish the existence theorem of Firth's modified estimate in the binomial logistic regression model, assuming only the full column rankness of the design matrix.
Our proof is consistent with the intuitive argument in \cite{Firth1993-li}.
We also discuss some binomial regression models other than the logistic regression model obtained through alternating link functions and derive similar existence results of the corresponding penalized maximum likelihood estimates for such models.

\section{Existence guarantee in the logistic regression case}
\label{sec:binomial}

Assume that we have $n$ realizations $y_1,\dots,y_n$ that are generated independently from $\Bin(m_i, \pi_i(\beta))$ ($i=1,\dots,n$).
Here, $m_i$ is a given positive integer and $\pi_i(\beta)$ is a probability determined by the following logistic model:
\begin{equation}
	\label{eq:logistic}
	\pi_i(\beta)
    = \frac{\exp(x_i^{\top}\beta)}{1 + \exp(x_i^{\top}\beta)},
\end{equation}
where $x_i=(x_{i1},\dots,x_{ip})^{\top}\in\R^p$ is the covariate vector of subject $i$ and $\beta=(\beta_1,\dots,\beta_p)^{\top}\in\R^p$ is a parameter vector.
Throughout this study, we assume that $n\ge p$.
The log-likelihood function is
\begin{align*}
    l(\beta)
    = \sum_{i=1}^n \left[y_i x_i^{\top}\beta - m_i\log\left\{1+\exp(x_i^{\top}\beta)\right\}\right]
\end{align*}
up to an additive constant irrelevant to $\beta$.
The Hessian matrix of $l(\beta)$ is $-X^{\top}MW(\beta)X$, where $X=(x_1,\dots,x_n)^{\top}\in\R^{n\times p}$, $M=\operatorname{diag}(m_1,\dots,m_n)$, and $W(\beta)=\operatorname{diag}\{w_1(\beta),\dots,w_n(\beta)\}$ with $w_i(\beta)=\pi_i(\beta)\{1-\pi_i(\beta)\}=\exp(x_i^{\top}\beta)/\{1+\exp(x_i^{\top}\beta)\}^2$.
Firth's penalized log-likelihood function is
\begin{align}
    \label{eq:pllk_b}
    l^*(\beta)
    := l(\beta) + \frac{1}{2}\log |X^{\top}MW(\beta)X|.
\end{align}
We define the corresponding estimator as the maximizer of the penalized log-likelihood, that is, a solution $\hat{\beta}\in\R^p$ satisfying
\begin{align}
	\label{eq:f_mle}
    l^*(\hat{\beta})
    = \sup_{\beta\in\R^p} l^*(\beta).
\end{align}
We refer to this estimator as Firth's modified estimator.

The main result of this study is the following theorem, which confirms the existence of Firth's modified estimates in the binomial logistic regression model.

\begin{theorem}
\label{the:existence_b}
If $X$ is of full column rank, there exists a maximizer $\hat{\beta}\in\R^p$ of $l^*$ and the set of maximizers $\argmax_{\beta\in\R^p} l^*(\beta)$ of $l^*$ is bounded.
\end{theorem}
\begin{proof}[Proof of Theorem \ref{the:existence_b}]
Let $c:=l^*(0)\in\R$.
From Lemma \ref{lem:bound_b}, we can find a constant $r_0>0$ such that $\sup_{u\in\R^p:\|u\|=1}l^*(ru)<c$ for all $r>r_0$.
Let $\cB(r_0):=\{\beta\in\R^p: \|\beta\|\le r_0\}$ be a closed ball with center $0$ and radius $r_0$.
The restriction $l^*:\cB(r_0)\to\R$ of $l^*(\beta)$ has a maximizer $\beta^*\in\cB(r_0)$, because $l^*(\beta)$ is continuous and $\cB(r_0)$ is compact.
Since $l^*(\beta)$ is lower than $c$ outside $\cB(r_0)$, $\beta^*$ is a maximizer of $l^*:\R^p\to\R$.
We see that the set $\argmax_{\beta\in\R^p} l^*(\beta)$ is bounded because any maximizer must be included in $\cB(r_0)$.
\end{proof}

\begin{lemma}
\label{lem:bound_b}
If $X$ has a full column rank, $\sup_{u\in\R^p : \|u\|=1}|X^{\top}MW(ru)X|\to0$ as $r\to\infty$.
\end{lemma}
\begin{proof}
Let $\cU:=\{u\in\R^p: \|u\|=1\}$ be the unit sphere.
For any $r>0$ and $u\in\cU$, we observe that
\begin{align}
	\label{eq:simple_ineq}
    w_i(ru)
    = \frac{1}{\{1+ \exp(rx_i^{\top}u)\}\{1+ \exp(-rx_i^{\top}u)\}}
    \le \frac{1}{1+ \exp(r|x_i^{\top}u|)}
    = \frac{1}{1+ \exp(r\|x_i\| \cdot |\cos \theta_i(u)|)},
\end{align}
where $\theta_i(u)$ denotes the angle between $x_i$ and $u$.

First, we consider the case $n=p$.
Let $w(r):=\sup_{u\in\cU}\prod_{i=1}^n w_i(ru)$ for $r>0$.
We then have
\begin{align*}
    w(r)
    &\le \sup_{u\in\cU} \prod_{i=1}^n \{1+ \exp(r\|x_i\| \cdot |\cos \theta_i(u)|)\}^{-1}\\
    &\le \sup_{u\in\cU} \prod_{i=1}^n \{1+ \exp(r a|\cos \theta_i(u)|)\}^{-1}\\
    &\le \sup_{u\in\cU} \prod_{i=1}^n \exp\{-ra |\cos \theta_i(u)|\}\\
    &= \exp\left[-ra \inf_{u\in\cU} \left\{\sum_{i=1}^n |\cos \theta_i(u)|\right\}\right],
\end{align*}
where $a:=\min\{\|x_i\|:i=1,\dots,n\}$.
Because $\cU$ is compact and $\sum_i|\cos\theta_i(u)|$ is continuous in $u$, a minimum $c:=\min_{u\in\cU} \sum_i|\cos\theta_i(u)|\ge0$ exists.
Since $X$ is of full column rank, we have $c>0$.
Therefore, it follows that $w(r) \le \exp(-acr) \to 0$ as $r\to\infty$.
Since $\sup_{u\in\cU}|X^{\top}MW(ru)X|=|X|^2|M|\sup_{u\in\cU}|W(ru)|=|X|^2|M|w(r)\ge0$, we obtain $\sup_{u\in\cU}|X^{\top}MW(ru)X|\to0$ as $r\to\infty$.

Subsequently, we consider the case $n>p$.
Using the Binet--Cauchy formula, $|X^{\top}MW(\beta)X|$ can be written as 
\begin{align*}
    |X^{\top}MW(\beta)X|
    &= \sum_{1\le i_1<i_2<\dots<i_p\le n} |X^{\top}|_{i_1,i_2,\dots,i_p}^{1,2,\dots,p} \cdot |MWX|^{i_1,i_2,\dots,i_p}_{1,2,\dots,p}\\
    &= \sum_{1\le i_1<i_2<\dots<i_p\le n} \left(|X|^{i_1,i_2,\dots,i_p}_{1,2,\dots,p}\right)^2 \prod_{k=1}^p m_{i_k} w_{i_k}(\beta)
    \ge 0,
\end{align*}
where $|A|_{j_1,j_2,\dots,j_p}^{1,2,\dots,p}$ is the minor determinant of a matrix $A\in\R^{n\times p}$ obtained by taking its $j_1,\dots,j_p$-th rows and $|B|^{j_1,j_2,\dots,j_p}_{1,2,\dots,p}$ is the minor determinant of a matrix $B\in\R^{p\times n}$ obtained by taking its $j_1,\dots,j_p$-th columns.
Using this expression, we can write as
\begin{align*}
    \sup_{u\in\cU} |X^{\top}MW(r u)X|
    &= \sup_{u\in\cU} \sum_{1\le i_1<i_2<\dots<i_p\le n} \left(|X|^{i_1,i_2,\dots,i_p}_{1,2,\dots,p}\right)^2 \prod_{k=1}^p m_{i_k} w_{i_k}(r u)\\
    &= \sum_{1\le i_1<i_2<\dots<i_p\le n} \left(|X|^{i_1,i_2,\dots,i_p}_{1,2,\dots,p}\right)^2 \left(\prod_{k=1}^p m_{i_k}\right) \sup_{u\in\cU} \prod_{k=1}^p w_{i_k}(ru).
\end{align*}
Using the discussion in the previous paragraph, for each term with $|X|^{i_1,i_2,\dots,i_p}_{1,2,\dots,p}\neq0$ on the right-hand side, we find a constant $c_{i_1,\dots,i_p}>0$ such that $\sup_u \prod_{k=1}^p w_{i_k}(r u) \le \exp(-c_{i_1,\dots,i_p}r)$.
Thus, letting $c^* := \min\{c_{i_1,\dots,i_p} : 1\le i_1<i_2<\dots<i_p\le n, ~ |X|^{i_1,i_2,\dots,i_p}_{1,2,\dots,p}\neq0\}$, it follows that
\begin{align*}
    \sup_{u\in\cU} |X^{\top}W(t u)X|
    \le \left\{\sum_{1\le i_1<i_2<\dots<i_p\le n} \left(|X|^{i_1,i_2,\dots,i_p}_{1,2,\dots,p}\right)^2 \prod_{k=1}^p m_{i_k}\right\}\exp(- c^*r)
    \to 0
    \quad \textnormal{as $r\to\infty$}.
\end{align*}
This completes the proof.
\end{proof}

To our knowledge, Theorem \ref{the:existence_b} is the first result that formally ensures the existence of Firth's modified estimate in logistic regression models.
As mentioned in Section \ref{sec:intro}, our proof of existence is consistent with the intuitive discussion in \cite{Firth1993-li} that claims the divergence of the penalty term to $-\infty$ when some entries of $\beta$ diverge.
\cite{Kosmidis2021-bk} also discussed the finiteness of the estimate in their Corollary 1, based on a weaker version of Lemma \ref{lem:bound_b}.
However, the uniformity of the bounds for the Jeffreys-prior penalty term was missing from their discussion.
Lemma \ref{lem:bound_b} fills this gap and allows us to formally establish an existence guarantee.

\section{Binomial regression models specified with other link functions}
\label{sec:other_binomial}

The logistic regression model \eqref{eq:logistic} is the binomial regression model specified by the logit link function in the framework of generalized linear models.
The probit and complementary log-log link functions are other common options of the link functions in binomial regression models.
For the corresponding models, we can consider the Jeffreys-prior penalized maximum likelihood estimators of the form \eqref{eq:f_mle}, where $l(\beta)$ and $W(\beta)$ in \eqref{eq:pllk_b} are replaced appropriately.
Specifically, the diagonal elements of $W(\beta)$ are
\begin{align*}
	w_i(\beta)
	= \begin{cases}
		\{\phi(z)\}^2/[\Phi(z)\{1-\Phi(z)\}] &\quad \textnormal{probit regression}\\
		\exp(-2z)/\{\exp(\exp(-z))-1\} &\quad \textnormal{complementary log-log regression},
	\end{cases}
\end{align*}
where $\phi(\cdot)$ and $\Phi(\cdot)$ are the cumulative distribution and probability density functions of the standard normal distribution, respectively.
Note that these penalized maximum likelihood estimators are different from the estimators obtained by Firth's method because the link functions are not canonical.
The following proposition yields the existence guarantee of the corresponding estimates.

\begin{proposition}
	\label{prop:probit-cloglog}
	For the probit and complementary log-log link functions the claims of Theorem \ref{the:existence_b} hold, that is, if $X$ is of full column rank, there exists a maximizer $\hat{\beta}\in\R^p$ of the corresponding $l^*$ and the set of maximizers $\argmax_{\beta\in\R^p} l^*(\beta)$ of $l^*$ is bounded.
\end{proposition}

\begin{proof}
	By using Lemmas \ref{lem:bound_probit} and \ref{lem:cloglog}, for the probit and complementary log-log regression models, we can find a positive constant $c$ such that
	\begin{equation*}
		w_i(ru)
		\le \frac{c}{1+\exp(r|x_i^{\top}u|)}
	\end{equation*}
	for any $r>0$ and $u\in\cU$.
	Using this inequation instead of \eqref{eq:simple_ineq}, the claim of Proposition \ref{prop:probit-cloglog} can be proved according to the proof of Theorem \ref{the:existence_b}.
\end{proof}

\section*{Acknowledgement}

This work was supported by JSPS KAKENHI Grant Number JP20K19752.

\appendix
\renewcommand{\thesection}{\appendixname\ \Alph{section}}
\renewcommand{\appendixname}{Appendix}
\renewcommand{\thetheorem}{\Alph{section}.\arabic{theorem}}
\setcounter{theorem}{0}

\section{Lemmas used in the proof of Proposition \ref{prop:probit-cloglog}}
\label{app:lemma}

The following lemmas are used in the proof of Proposition \ref{prop:probit-cloglog}.

\begin{lemma}
	\label{lem:bound_probit}
	There exists a constant $c\in\R$ such that 
	\begin{equation*}
		f(z):=
		(1+e^{|z|}) \frac{\{\phi(z)\}^2}{\Phi(z)\{1-\Phi(z)\}} < c, \quad z\in\R.
	\end{equation*}
\end{lemma}
\begin{proof}
	Since the function on the left-hand side is even, we assume $z>0$ without loss of generality.
	Because $\Phi(z)>1/2$ and $1-\Phi(z)>\frac{z}{z^2+1}\phi(z)$, we have
	\begin{equation*}
		f(z)
		< (1+e^{|z|}) \phi(z) \frac{2(z^2+1)}{z} \to 0.
	\end{equation*}
	Then, for any $\varepsilon>0$ we can take $\delta>0$ such that $f(z)<\varepsilon$ for $z>\delta$.
	Therefore we have $f(z)\le\max\{ \max_{s\in[0,\delta]}f(s), \varepsilon\}<\infty$.
\end{proof}

\begin{lemma}
	\label{lem:cloglog}
	There exists a constant $c\in\R$ such that 
	\begin{equation*}
		f(z):=
		(1+e^{|z|}) \frac{\exp(-2z)}{\exp(\exp(-z))-1} < c, \quad z\in\R.
	\end{equation*}
\end{lemma}
\begin{proof}
	Since $e^z > 1+z+z^2/2+z^3/6$ for $z>0$, we have
	\begin{equation*}
		(1+e^{|z|}) \frac{\exp(-2z)}{\exp(\exp(-z))-1}
		< (1+e^{|z|}) \frac{e^{-2z}}{e^{-z} + e^{-2z}/2 + e^{-3z}/6}
		= (1+e^{|z|}) \frac{e^{-z}}{1 + e^{-z}/2 + e^{-2z}/6}
		=: g(z). 
	\end{equation*}
	For $z>0$, 
	\begin{equation*}
		g(z)
		= \frac{1+e^{-z}}{1 + e^{-z}/2 + e^{-2z}/6}
		\to 1 \quad \textnormal{as $z\to\infty$}.
	\end{equation*}
	For $z<0$,
	\begin{equation*}
		g(z)
		= \frac{1 + e^z}{e^{2z} + e^z/2 + 1/6}
		\to 6 \quad \textnormal{as $z\to-\infty$}.
	\end{equation*}
	Thus, for any $\varepsilon<0$, there exists $\delta>0$ such that $g(z)<6 + \varepsilon$ for any $|z|>\delta$.
	Therefore we have $f(z)<g(z)\le\max\{6 + \varepsilon, \max_{s: |s|\le\delta}g(s)\}<\infty$.
\end{proof}

\bibliographystyle{jasa}

\end{document}

%% file: math.tex
\newcommand{\R}{{\mathbb R}}

\newcommand{\cB}{{\mathcal B}}

\newcommand{\cU}{{\mathcal U}}

\bmdefine{\Bzero}{0}
\bmdefine{\Bone}{1}
\bmdefine{\Ba}{a}
\bmdefine{\Bb}{b}
\bmdefine{\Bc}{c}
\bmdefine{\Bd}{d}
\bmdefine{\Be}{e}
\bmdefine{\Bf}{f}
\bmdefine{\Bg}{g}
\bmdefine{\Bh}{h}
\bmdefine{\Bi}{i}
\bmdefine{\Bj}{j}
\bmdefine{\Bk}{k}
\bmdefine{\Bl}{l}
\bmdefine{\Bell}{\ell}
\bmdefine{\Bm}{m}
\bmdefine{\Bn}{n}
\bmdefine{\Bp}{p}
\bmdefine{\Bq}{q}
\bmdefine{\Br}{r}
\bmdefine{\Bs}{s}
\bmdefine{\Bt}{t}
\bmdefine{\Bu}{u}
\bmdefine{\Bv}{v}
\bmdefine{\Bw}{w}
\bmdefine{\Bx}{x}
\bmdefine{\By}{y}
\bmdefine{\Bz}{z}
\bmdefine{\BA}{A}
\bmdefine{\BB}{B}
\bmdefine{\BC}{C}
\bmdefine{\BD}{D}
\bmdefine{\BF}{F}
\bmdefine{\BG}{G}
\bmdefine{\BH}{H}
\bmdefine{\BI}{I}
\bmdefine{\BM}{M}
\bmdefine{\BP}{P}
\bmdefine{\BS}{S}
\bmdefine{\BT}{T}
\bmdefine{\BR}{R}
\bmdefine{\BV}{V}
\bmdefine{\BW}{W}
\bmdefine{\BX}{X}
\bmdefine{\BY}{Y}
\bmdefine{\BZ}{Z}

\bmdefine{\Balp}{\alpha}
\bmdefine{\Bbet}{\beta}
\bmdefine{\Bdelta}{\delta}
\bmdefine{\BDelta}{\Delta}
\bmdefine{\Beps}{\epsilon}
\bmdefine{\Beta}{\eta}
\bmdefine{\Bmu}{\mu}
\bmdefine{\Bnu}{\nu}
\bmdefine{\Bpi}{\pi}
\bmdefine{\Bpsi}{\psi}
\bmdefine{\BPsi}{\Psi}
\bmdefine{\Btau}{\tau}
\bmdefine{\Btheta}{\theta}
\bmdefine{\BTheta}{\Theta}
\bmdefine{\Bpartial}{\partial}
\bmdefine{\Bsigma}{\sigma}
\bmdefine{\Bveps}{\varepsilon}
\bmdefine{\Bxi}{\xi}

\newcommand{\argmax}{\operatorname{argmax}}

\newcommand{\Bin}{\mathop{\operatorname{Bin}}}